\documentclass[a4paper,oneside,11pt]{article}%
\usepackage{makeidx}
\usepackage[english]{babel}
\usepackage{amsmath}
\usepackage{amsfonts}
\usepackage{amssymb}
\usepackage{stmaryrd}
\usepackage{graphicx}
\usepackage{mathrsfs}
\usepackage[colorlinks,linkcolor=red,anchorcolor=blue,citecolor=blue,urlcolor=blue]{hyperref}
\usepackage[symbol*,ragged]{footmisc}

\allowdisplaybreaks[4]
\providecommand{\U}[1]{\protect\rule{.1in}{.1in}}
\providecommand{\U}[1]{\protect \rule{.1in}{.1in}}

\newtheorem{theorem}{Theorem}[section]

\newtheorem{lemma}[theorem]{Lemma}

\newtheorem{proposition}[theorem]{Proposition}

\newenvironment{proof}[1][Proof]{\noindent \textbf{#1.} }{\  \rule{0.5em}{0.5em}}
\numberwithin{equation}{section}

\begin{document}

\title{On a Diophantine equation with five prime variables }

\author{Min Zhang\footnotemark[1]\,\,\,\, \, \& \,\,Jinjiang Li\footnotemark[2] \vspace*{-4mm} \\
$\textrm{\small Department of Mathematics, China University of Mining and Technology}^{*\,\dag}$
                    \vspace*{-4mm} \\
     \small  Beijing 100083, P. R. China  }

\footnotetext[2]{Corresponding author. \\
    \quad\,\, \textit{ E-mail addresses}:
     \href{mailto:min.zhang.math@gmail.com}{min.zhang.math@gmail.com} (M. Zhang),
     \href{mailto:jinjiang.li.math@gmail.com}{jinjiang.li.math@gmail.com} (J. Li).   }

\date{}
\maketitle

{\textbf{Abstract}}: Let $[x]$ denote the integral part of the real number $x$, and $N$ be a sufficiently large integer. In this paper, it is proved that,
for $1<c<\frac{4109054}{1999527}, c\not=2$, the Diophantine equation $N=[p_1^c]+[p_2^c]+[p_3^c]+[p_4^c]+[p_5^c]$ is solvable in prime variables
$p_1,p_2,p_3,p_4,p_5$.


{\textbf{MR(2010) Subject Classification}}: 11P05, 11P32, 11L07, 11L20

\section{Introduction and main result}
Let $[x]$ be the integral part of the real number $x$. In 1933--1934, Segal \cite{Segal-1933-1,Segal-1933-2} first considered the Waring's problem with non--integer degrees, who showed that for any sufficiently large integer $N$ and $c>1$ being not an integer, there exists a  integer $k_0=k_0(c)>0$ such that the equation
\begin{equation*}
N=[x_1^c]+[x_2^c]+\cdots+[x_k^c]
\end{equation*}
is solvable for $k\geqslant k_0(c)$. Later, Segal's  bound for $k_0(c)$ was improved by Deshouillers \cite{Deshouillers-1973} and by
Arkhilov and Zhitkov \cite{Arkhipov-Zhitkov-1984}, respectively. Let $G(c)$ be the least of the integers $k_0(c)$ such that every sufficiently large
integer $N$ can be written as a sum of not more than $k_0(c)$ numbers with the form $[n^c]$. In particular, Deshouillers \cite{Deshouillers-1974} and
Gritsenko \cite{Gritsenko-1993} considered the case $k=2$ and gave $G(c)=2$ for $1<c<4/3$ and $1<c<55/41$, respectively.

In 1937, Vinogradov \cite{Vinogradov-1937} solved asymptotic form of the ternary Goldbach problem. He proved that, for sufficiently large integer $N$
satisfying $N\equiv1\pmod2$, the following equation
\begin{equation*}
N=p_1+p_2+p_3
\end{equation*}
is solvable in primes $p_1,p_2,p_3$. As an analogue of the ternary Goldbach problem, in 1995, Laporta and Tolev \cite{Laporta-Tolev-1995} investigated
the solvability of the following equation
\begin{equation*}
    N=[p_1^c]+[p_2^c]+[p_3^c]
\end{equation*}
in prime variables $p_1,p_2,p_3$. Define
\begin{equation*}
\mathscr{R}_s(N)=\sum_{N=[p_1^c]+[p_2^c]+\cdots+[p_s^c]}(\log p_1)(\log p_2)\cdots(\log p_s).
\end{equation*}
Laporta and Tolev \cite{Laporta-Tolev-1995} showed that the sum $\mathscr{R}_3(N)$ has asymptotic formula for $1<c<17/16$ and gave
\begin{equation*}
   \mathscr{R}_3(N)=\frac{\Gamma^3(1+1/c)}{\Gamma(3/c)}N^{3/c-1}+O\Big(N^{3/c-1}\exp\big(-(\log N)^{1/3-\delta}\big)\Big)
\end{equation*}
for any $0<\delta<1/3$. Later, Kumchev and Nedeva \cite{Kumchev-Nedeva-1998} improved the result of Laporta and Tolev \cite{Laporta-Tolev-1995}, and enlarged the range of $c$ to $12/11$. Afterwards, Zhai and Cao \cite{Zhai-Cao-2002} refined the result of Kumchev and Nedeva \cite{Kumchev-Nedeva-1998}, who extended
the range of $c$ to $258/235$. In 2018, Cai \cite{Cai-2018} enhanced the result of Zhai and Cao \cite{Zhai-Cao-2002} and gave the upper bound of $c$ as
$137/119$.

 In 1938, Hua \cite{Hua-1938} proved that every sufficiently large integer $N$, which satisfies $N\equiv5\pmod{24}$, can be represented as five squares of primes, i.e.,
\begin{equation*}
   N=p_1^2+p_2^2+p_3^2+p_4^2+p_5^2.
\end{equation*}

In this paper, as an analogue of Hua's five square theorem, we shall investigate the solvability of the following Diophantine equation
\begin{equation*}
    N=[p_1^c]+[p_2^c]+[p_3^c]+[p_4^c]+[p_5^c]
\end{equation*}
in prime variables $p_1,p_2,p_3,p_4,p_5$, and devote to establish the following result.

\begin{theorem}\label{Theorem-five-variables-integral}
   Let $1<c<\frac{4109054}{1999527},c\not=2$, and $N$ be a sufficiently large integer. Then we have
\begin{equation*}
   \mathscr{R}_5(N)=\frac{\Gamma^5(1+1/c)}{\Gamma(5/c)}N^{5/c-1}+O\Big(N^{5/c-1}\exp\big(-(\log N)^{1/4}\big)\Big),
\end{equation*}
where the implied constant in the $O$--term depends only on $c$.
\end{theorem}

\smallskip
\textbf{Notation.}
Throughout this paper, we suppose that $1<c<\frac{4109054}{1999527},c\not=2$. Let $p$, with or without subscripts, always denote a prime number; $\varepsilon$
always denote arbitrary small positive constant, which may not be the same at different occurrences. As usual, we use $[x],\,\{x\}$ and $\|x\|$
to denote the integral part of $x$, the fractional part of $x$ and the distance from $x$ to the nearest integer, respectively.
Also, we write $e(x)=e^{2\pi i x}$; $f(x)\ll g(x)$ means that $f(x)=O(g(x))$; $f(x)\asymp g(x)$ means that $f(x)\ll g(x)\ll f(x)$.

We also define
\begin{align*}
 & \,\,  P=N^{1/c},\qquad \quad\tau=P^{1-c-\varepsilon}, \qquad \quad S(\alpha)=\sum_{p\leqslant P}(\log p)e\big([p^{c}]\alpha\big),   \\
 & \,\, \mathcal{T}(\alpha,X)=\sum_{X<n\leqslant2X}e\big([n^{c}]\alpha\big),\qquad S(\alpha,X)=\sum_{X<p\leqslant2X}(\log p)e\big([p^{c}]\alpha\big).
\end{align*}

\section{Preliminary Lemmas}
In this section, we shall state some preliminary lemmas, which are required in the  proof of Theorem \ref{Theorem-five-variables-integral}.

\begin{lemma}\label{Titchmarsh-lemma4.8}
Let $f(x)$ be a real differentiable function in the interval $[a,b]$. If $f'(x)$ is monotonic and satisfies $|f'(x)|\leqslant\theta<1$. Then we have
\begin{equation*}
   \sum_{a<n\leqslant b}e^{2\pi if(n)}=\int_a^be^{2\pi if(x)}\mathrm{d}x+O(1).
\end{equation*}
\end{lemma}
\begin{proof}
 See Lemma 4.8 of Titchmarsh  \cite{Titchmarsh-book}  .  $\hfill$
\end{proof}

\begin{lemma}\label{Fouvry-Iwaniec-chafen-lemma}
Let $L,Q\geqslant1$ and $z_\ell$ be complex numbers. Then we have
\begin{equation*}
 \Bigg|\sum_{L<\ell\leqslant2L}z_\ell\Bigg|^2\leqslant\bigg(2+\frac{L}{Q}\bigg)\sum_{|q|<Q}\bigg(1-\frac{|q|}{Q}\bigg)
 \sum_{L<\ell+q,\ell-q\leqslant2L}z_{\ell+q}\overline{z_{\ell-q}}.
\end{equation*}
\end{lemma}
\begin{proof}
 See Lemma 2 of Fouvry and Iwaniec \cite{Fouvry-Iwaniec-1989}.  $\hfill$
\end{proof}

\begin{lemma}\label{yijie-exp-pair}
Suppose that $f(x):[a,b]\to\mathbb{R}$ has continuous derivatives of arbitrary order on $[a,b]$, where $1\leqslant a<b\leqslant2a$. Suppose further that
\begin{equation*}
 \big|f^{(j)}(x)\big|\asymp \lambda_1 a^{1-j},\qquad j\geqslant1, \qquad x\in[a,b].
\end{equation*}
Then for any exponential pair $(\kappa,\lambda)$, we have
\begin{equation*}
 \sum_{a<n\leqslant b}e(f(n))\ll \lambda_1^\kappa a^\lambda+\lambda_1^{-1}.
\end{equation*}
\end{lemma}
\begin{proof}
 See (3.3.4) of Graham and Kolesnik \cite{Graham-Kolesnik-book}.  $\hfill$
\end{proof}

\begin{lemma}\label{Buriev-exp-zhankai}
   Let $x$ be not an integer, $\alpha\in(0,1),\, H\geqslant 3$. Then we have
\begin{equation*}
 e(-\alpha \{x\})=\sum_{|h|\leqslant H} c_h(\alpha)e(hx)+O\left( \min \bigg( 1, \frac {1}{H\|x\|}\bigg)\right),
\end{equation*}
where
\begin{equation*}
c_h(\alpha)=\frac {1-e(-\alpha)}{2\pi i(h+\alpha)}.
\end{equation*}
\end{lemma}
\begin{proof}
 See Lemma 12 of Buriev \cite{Buriev-thesis-1989} or Lemma 3 of Kumchev and Nedeva \cite{Kumchev-Nedeva-1998}.  $\hfill$
\end{proof}

\begin{lemma}\label{Robert-lemma}
  Suppose $Y>1,\gamma>0, c>1, c\not\in\mathbb{Z}$. Let $\mathscr{A}(Y;c,\gamma)$ denote the number of solutions of the inequality
 \begin{eqnarray*}
   \big|n_1^c+n_2^c-n_3^c-n_4^c\big|<\gamma,\qquad Y<n_1,n_2,n_3,n_4\leqslant2Y,
 \end{eqnarray*}
 then
 \begin{equation*}
    \mathscr{A}(Y;c,\gamma)\ll(\gamma Y^{4-c}+Y^2)Y^\varepsilon.
 \end{equation*}
\end{lemma}
\begin{proof}
 See Theorem 2 of Robert and Sargos \cite{Robert-Sargos-2006}.  $\hfill$
\end{proof}

\begin{lemma}\label{S(alpha)-fourth-power}
  For $1<c<3,c\not=2$, we have
 \begin{equation*}
   \int_0^1\big|S(\alpha)\big|^4\mathrm{d}\alpha\ll (P^{4-c}+P^2)P^\varepsilon.
 \end{equation*}
\end{lemma}
\begin{proof}
 By a splitting argument, it is sufficient to show that
\begin{equation*}
  \int_0^1\big|S(\alpha,P/2)\big|^4\mathrm{d}\alpha\ll(P^{4-c}+P^2)P^\varepsilon.
\end{equation*}
Trivially, we have
\begin{align*}
     &  \int_0^1\big|S(\alpha,P/2)\big|^4\mathrm{d}\alpha
                       \nonumber  \\
   = & \sum_{P/2<p_1,p_2,p_3,p_4\leqslant P}(\log p_1)\cdots(\log p_4)\int_0^1e\Big(\big([p_1^c]+[p_2^c]-[p_3^c]-[p_4^c]\big)\alpha\Big)\mathrm{d}\alpha
                     \nonumber  \\
   = & \sum_{\substack{P/2<p_1,p_2,p_3,p_4\leqslant P \\ [p_1^c]+[p_2^c]=[p_3^c]+[p_4^c] }} (\log p_1) \cdots(\log p_4)
         \ll  (\log P)^4\sum_{\substack{P/2<n_1,n_2,n_3,n_4\leqslant P \\ [n_1^c]+[n_2^c]=[n_3^c]+[n_4^c] }}1.
\end{align*}
On the other hand, if $[n_1^c]+[n_2^c]=[n_3^c]+[n_4^c]$, we can deduce that
\begin{equation*}
  \big|n_1^c+n_2^c-n_3^c-n_4^c\big|=\big|\{n_1^c\}+\{n_2^c\}-\{n_3^c\}-\{n_4^c\}\big|\leqslant2.
\end{equation*}
By Lemma \ref{Robert-lemma}, we derive that
\begin{equation*}
  \int_0^1\big|S(\alpha,P/2)\big|^4\mathrm{d}\alpha\ll (\log P)^4\cdot\mathscr{A}(P/2;c,2)\ll (P^{4-c}+P^2)P^\varepsilon,
\end{equation*}
which completes the proof of Lemma \ref{S(alpha)-fourth-power}.   $\hfill$
\end{proof}

\begin{lemma}\label{S^*(alpha)-fourth-power}
  For $1<c<3,c\not=2$, we have
 \begin{equation*}
   \int_{-\tau}^\tau\big|S(\alpha)\big|^4\mathrm{d}\alpha\ll P^{4-c}\log^6P.
 \end{equation*}
\end{lemma}
\begin{proof}
By a splitting argument, it is sufficient to show that
\begin{equation}\label{S^*(alpha)-fourth-power-sufficient}
   \int_{-\tau}^\tau\big|S(\alpha,P/2)\big|^4\mathrm{d}\alpha\ll P^{4-c}\log^5P.
\end{equation}
We have
\begin{align}\label{S^*(alpha)-fourth-power-upper}
    &      \int_{-\tau}^\tau\big|S(\alpha,P/2)\big|^4\mathrm{d}\alpha
               \nonumber \\
  = & \sum_{P/2<p_1,p_2,p_3,p_4\leqslant P}(\log p_1)\cdots(\log p_4)\int_{-\tau}^{\tau}e\big(([p_1^c]+[p_2^c]-[p_3^c]-[p_4^c])\alpha\big)\mathrm{d}\alpha
            \nonumber \\
  \ll & \sum_{P/2<p_1,p_2,p_3,p_4\leqslant P}(\log p_1)\cdots(\log p_4)\min\bigg(\tau,\frac{1}{\big|[p_1^c]+[p_2^c]-[p_3^c]-[p_4^c]\big|}\bigg)
             \nonumber \\
  \ll &\,\,  \mathscr{U}\tau\log^4P+\mathscr{V}\log^4P,
\end{align}
where
\begin{equation*}
  \mathscr{U}=\sum_{\substack{P/2<n_1,n_2,n_3,n_4\leqslant P\\ |[n_1^c]+[n_2^c]-[n_3^c]-[n_4^c]|\leqslant1/\tau}}1,\qquad
  \mathscr{V}=\sum_{\substack{P/2<n_1,n_2,n_3,n_4\leqslant P\\ |[n_1^c]+[n_2^c]-[n_3^c]-[n_4^c]|>1/\tau}}\frac{1}{\big|[n_1^c]+[n_2^c]-[n_3^c]-[n_4^c]\big|}.
\end{equation*}
We have
\begin{align*}
\mathscr{U} \ll & \sum_{P/2<n_1\leqslant P}\sum_{P/2<n_2\leqslant P}\sum_{P/2<n_3\leqslant P}
                  \sum_{\substack{P/2<n_4\leqslant P\\ ([n_1^c]+[n_2^c]-[n_3^c]-1/\tau)^{1/c}\leqslant n_4\leqslant([n_1^c]+[n_2^c]-[n_3^c]+1/\tau+1)^{1/c} \\
                          [n_1^c]+[n_2^c]-[n_3^c]\asymp P^c}}1
                            \nonumber \\
      \ll & \sum_{\substack{P/2<n_1,n_2,n_3\leqslant P\\ [n_1^c]+[n_2^c]-[n_3^c]\asymp P^c}}
            \Big(1+\big([n_1^c]+[n_2^c]-[n_3^c]+1/\tau+1\big)^{\frac{1}{c}}-\big([n_1^c]+[n_2^c]-[n_3^c]-1/\tau\big)^{\frac{1}{c}}\Big),
\end{align*}
and by the mean--value theorem
\begin{equation}
  \mathscr{U}\ll P^3+\frac{1}{\tau}P^{4-c}.
\end{equation}
Obviously, $ \mathscr{V}\leqslant\sum\limits_{\ell} \mathscr{V}_\ell$, where
\begin{equation}
  \mathscr{V}_\ell =\sum_{\substack{P/2<n_1,n_2,n_3,n_4\leqslant P  \\
                   \ell<|[n_1^c]+[n_2^c]-[n_3^c]-[n_4^c]|\leqslant2\ell}}\frac{1}{\big|[n_1^c]+[n_2^c]-[n_3^c]-[n_4^c]\big|},
\end{equation}
and $\ell$ takes the values $\frac{2^k}{\tau},\,k=0,1,2,\dots$, with $\ell\ll P^c$. Then, we derive that
\begin{equation*}
  \mathscr{V}_\ell \ll \frac{1}{\ell}\sum_{\substack{P/2<n_1,n_2,n_3,n_4\leqslant P\\
                         ([n_1^c]+[n_2^c]-[n_3^c]+\ell)^{1/c}\leqslant n_4\leqslant([n_1^c]+[n_2^c]-[n_3^c]+2\ell+1)^{1/c}  \\
                         [n_1^c]+[n_2^c]-[n_3^c]\asymp P^c}}1.
\end{equation*}
For $\ell\geqslant\frac{1}{\tau}$ and $P/2<n_1,n_2,n_3\leqslant P$ with $[n_1^c]+[n_2^c]-[n_3^c]\asymp P^c$, it is easy to see that
\begin{equation*}
  ([n_1^c]+[n_2^c]-[n_3^c]+2\ell+1)^{1/c}-([n_1^c]+[n_2^c]-[n_3^c]+\ell)^{1/c}>1.
\end{equation*}
Hence, by the mean--value theorem, we get
\begin{equation}\label{V_l-upper}
  \mathscr{V}_\ell\ll \frac{1}{\ell}\sum_{\substack{P/2<n_1,n_2,n_3\leqslant P\\ [n_1^c]+[n_2^c]-[n_3^c]\asymp P^c}}
  \Big(([n_1^c]+[n_2^c]-[n_3^c]+2\ell+1)^{1/c}-([n_1^c]+[n_2^c]-[n_3^c]+\ell)^{1/c}\Big)\ll P^{4-c}.
\end{equation}
Combining (\ref{S^*(alpha)-fourth-power-upper})--(\ref{V_l-upper}), we obtain the desired estimate (\ref{S^*(alpha)-fourth-power-sufficient}), which
completes the proof of Lemma \ref{S^*(alpha)-fourth-power}.  $\hfill$
\end{proof}

\begin{lemma}\label{Heath-Brown-exponent-sum-fenjie}
Let $3<U<V<Z<X$ and suppose that $Z-\frac{1}{2}\in\mathbb{N},\,X\gg Z^2U,\,Z\gg U^2,\,V^3\gg X$. Assume further that $F(n)$ is a
complex--valued function such that $|F(n)|\leqslant1$. Then the sum
\begin{equation*}
 \sum_{X<n\leqslant2X}\Lambda(n)F(n)
\end{equation*}
may be decomposed into $O(\log^{10}X)$ sums, each of which either of Type I:
\begin{equation*}
 \sum_{M<m\leqslant2M}a(m)\sum_{K<k\leqslant2K}F(mk)
\end{equation*}
with $K\gg Z$, where $a(m)\ll m^{\varepsilon},\,MK\asymp X$, or of Type II:
\begin{equation*}
 \sum_{M<m\leqslant2M}a(m)\sum_{K<k\leqslant2K}b(k)F(mk)
\end{equation*}
with $U\ll M\ll V$, where $a(m)\ll m^{\varepsilon},\,b(k)\ll k^{\varepsilon},\,MK\asymp X$.
\end{lemma}
\begin{proof}
 See Lemma 3 of Heath--Brown \cite{Heath-Brown-1983}.  $\hfill$
\end{proof}

\begin{lemma}\label{Heath-Brown-extention}
For any real number $\theta$, there holds
\begin{equation*}
 \min\bigg(1,\frac{1}{H\|\theta\|}\bigg)=\sum_{h=-\infty}^{+\infty}a_he(h\theta),
\end{equation*}
where
\begin{equation*}
 a_h\ll\min\bigg(\frac{\log2H}{H},\frac{1}{|h|},\frac{H}{h^2}\bigg).
\end{equation*}
\end{lemma}
\begin{proof}
 See p.245 of Heath--Brown \cite{Heath-Brown-1983}.  $\hfill$
\end{proof}

\begin{lemma}\label{type-1-three-variables-integral-5}
Let $1<c<\frac{4109054}{1999527},c\not=2, P^{\frac{9449}{10000}}\ll X\ll P, H=X^{\frac{38687}{2666036}}$ and $c_h(\alpha)$ denote complex numbers such that $|c_h(\alpha)|\ll(1+|h|)^{-1}$. Then, for any $\alpha\in(\tau,1-\tau)$, if $M\ll X^{\frac{1371705}{2666036}}$, we have
\begin{equation*}
 \mathcal{S}_I(\alpha):=\sum_{|h|\leqslant H}c_h(\alpha)\sum_{M<m\leqslant2M}a(m)\sum_{K<k\leqslant2K}e\big((h+\alpha)(mk)^c\big)
                        \ll X^{\frac{2627349}{2666036}+\varepsilon},
\end{equation*}
where $a(m)\ll m^\varepsilon$ and $MK\asymp X$.
\end{lemma}
\begin{proof}
Obviously, we have
\begin{equation}\label{S_1-trivial-upper-5}
\big|\mathcal{S}_I(\alpha)\big|\ll X^\varepsilon \max_{|\xi|\in(\tau,H+1)}\sum_{M<m\leqslant2M}\Bigg|\sum_{K<k\leqslant2K}e\big(\xi(mk)^c\big)\Bigg|.
\end{equation}
Then we use Lemma \ref{yijie-exp-pair} to estimate the inner sum over $k$ in (\ref{S_1-trivial-upper-5}) with exponential
pair $(\kappa,\lambda)$ and derive that
\begin{align*}
  \mathcal{S}_I(\alpha)\ll &\,\,  X^\varepsilon \max_{|\xi|\in(\tau,H+1)}
                             \sum_{M<m\leqslant2M}\bigg((|\xi|X^cK^{-1})^{\kappa}K^{\lambda}+\frac{K}{|\xi|X^c}\bigg)
                                  \nonumber \\
   \ll & \,\,  X^\varepsilon \max_{|\xi|\in(\tau,H+1)} \bigg(|\xi|^\kappa X^{\kappa c}K^{\lambda-\kappa}M+\frac{MK}{|\xi|X^c}\bigg)
                                  \nonumber \\
   \ll & \,\,X^\varepsilon \big(H^{\kappa}X^{\kappa c+\lambda-\kappa}M^{\kappa+1-\lambda}+X^{1-c}\tau^{-1}\big),
\end{align*}
By taking
\begin{align*}
 (\kappa,\lambda) = & \,\, A^3 BABABABABABAB(0,1)
    = \bigg(\frac{33}{1550},\frac{698}{775}\bigg),
\end{align*}
we can see that, if $M\ll X^{\frac{1371705}{2666036}}$, then there holds
\begin{equation*}
   \mathcal{S}_I(\alpha)\ll X^{\frac{2627349}{2666036}+\varepsilon},
\end{equation*}
which completes the proof of Lemma \ref{type-1-three-variables-integral-5}.     $\hfill$
\end{proof}

\begin{lemma}\label{type-2-three-variables-integral-5}
Let $1<c<\frac{4109054}{1999527},c\not=2, P^{\frac{9449}{10000}}\ll X\ll P, H=X^{\frac{38687}{2666036}}$ and $c_h(\alpha)$ denote complex numbers such that
$|c_h(\alpha)|\ll(1+|h|)^{-1}$. Then, for any $\alpha\in(\tau,1-\tau)$, if there holds $X^{\frac{38687}{1333018}}\ll M\ll X^{\frac{11958325}{23994324}}$,
then we have
\begin{equation*}
 \mathcal{S}_{II}(\alpha):=\sum_{|h|\leqslant H}c_h(\alpha)\sum_{M<m\leqslant2M}a(m)\sum_{K<k\leqslant2K}b(k)e\big((h+\alpha)(mk)^c\big)
                        \ll X^{\frac{2627349}{2666036}+\varepsilon},
\end{equation*}
where $a(m)\ll m^\varepsilon,\,b(k)\ll k^\varepsilon$ and $MK\asymp X$.
\end{lemma}
\begin{proof}
 Let $Q=X^{\frac{38687}{1333018}}(\log X)^{-1}$. From Lemma \ref{Fouvry-Iwaniec-chafen-lemma} and Cauchy's inequality, we derive that
\begin{align}\label{type-2-Cauchy-chafen-5}
       &  \big|\mathcal{S}_{II}(\alpha)\big|
   \ll \,\, \sum_{|h|\leqslant H}\big|c_h(\alpha)\big|\Bigg|\sum_{K<k\leqslant2K}b(k)\sum_{M<m\leqslant2M}a(m)e\big((h+\alpha)(mk)^c\big)\Bigg|
                    \nonumber \\
  \ll & \,\, \sum_{|h|\leqslant H}\big|c_h(\alpha)\big|\Bigg(\sum_{K<k\leqslant2K}|b(k)|^2\Bigg)^{\frac{1}{2}}
            \Bigg(\sum_{K<k\leqslant2K}\bigg|\sum_{M<m\leqslant2M}a(m)e\big((h+\alpha)(mk)^c\big)\bigg|^2\Bigg)^{\frac{1}{2}}
                    \nonumber \\
  \ll & \,\, K^{\frac{1}{2}+\varepsilon}\sum_{|h|\leqslant H}\big|c_h(\alpha)\big|\Bigg(\sum_{K<k\leqslant2K}\frac{M}{Q}\sum_{0\leqslant q<Q}\bigg(1-\frac{q}{Q}\bigg)
                     \nonumber \\
   & \,\,\qquad\qquad \qquad\qquad\quad\times\sum_{M+q<m\leqslant2M-q}a(m+q)\overline{a(m-q)}e\big((h+\alpha)k^c\Delta_c(m,q)\big)\Bigg)^{\frac{1}{2}}
                     \nonumber \\
   \ll & \,\, K^{\frac{1}{2}+\varepsilon}\sum_{|h|\leqslant H}\big|c_h(\alpha)\big|
              \Bigg(\frac{M}{Q}\sum_{K<k\leqslant2K}\bigg(M^{1+\varepsilon}+\sum_{1\leqslant q<Q}\bigg(1-\frac{q}{Q}\bigg)
                      \nonumber \\
       & \,\, \quad\qquad\qquad\qquad\quad\times
                   \sum_{M+q<m\leqslant2M-q}a(m+q)\overline{a(m-q)}e\big((h+\alpha)k^c\Delta_c(m,q)\big)\bigg)\Bigg)^{\frac{1}{2}}
                      \nonumber \\
   \ll & \,\, X^{\varepsilon}\sum_{|h|\leqslant H}\big|c_h(\alpha)\big|
              \Bigg(\frac{X^2}{Q}+\frac{X}{Q}\sum_{1\leqslant q<Q}\sum_{M<m\leqslant2M}
                    \bigg|\sum_{K<k\leqslant2K}e\big((h+\alpha)k^c\Delta_c(m,q)\big)\bigg|\Bigg)^{\frac{1}{2}},
\end{align}
 where $\Delta_c(m,q)=(m+q)^c-(m-q)^c$. Thus, it is sufficient to estimate the sum
\begin{equation*}
   S_0:=\sum_{K<k\leqslant2K}e\big((h+\alpha)k^c\Delta_c(m,q)\big).
\end{equation*}
By Lemma \ref{yijie-exp-pair} with the exponential pair $(\kappa,\lambda)=A^2B(0,1)=(\frac{1}{14},\frac{11}{14})$, we have
\begin{equation*}
   S_0\ll \big(|h+\alpha|X^{c-1}q\big)^{\frac{1}{14}}K^{\frac{11}{14}}+\frac{1}{|h+\alpha|X^{c-1}q}.
\end{equation*}
Putting the above estimate into (\ref{type-2-Cauchy-chafen-5}), we obtain that
\begin{align*}
           \mathcal{S}_{II}(\alpha)
 \ll &\,\,  X^\varepsilon \sum_{|h|\leqslant H}\big|c_h(\alpha)\big|
            \Bigg(\frac{X^2}{Q}+\frac{X}{Q}\sum_{1\leqslant q<Q}\sum_{M<m\leqslant2M}
                       \nonumber \\
     &\,\, \qquad \qquad\qquad \qquad
            \times\bigg(\big(|h+\alpha|X^{c-1}q\big)^{\frac{1}{14}}K^{\frac{11}{14}}+\frac{1}{|h+\alpha|X^{c-1}q}\bigg)\Bigg)^{\frac{1}{2}}
                       \nonumber \\
 \ll & \,\, X^{\varepsilon}\sum_{|h|\leqslant H}\big|c_h(\alpha)\big|
             \bigg(\frac{X^2}{Q}+\frac{X}{Q}\Big(H^{\frac{1}{14}}X^{\frac{1}{14}(c-1)}MK^{\frac{11}{14}}Q^{\frac{15}{14}}
                    +X^{1-c}M\tau^{-1}\log Q\Big)\bigg)^{\frac{1}{2}}
                       \nonumber \\
 \ll & \,\, X^{1+\varepsilon}Q^{-\frac{1}{2}}\sum_{|h|\leqslant H}\big|c_h(\alpha)\big|
            \ll  X^{1+\varepsilon}Q^{-\frac{1}{2}}\sum_{|h|\leqslant H}\frac{1}{1+|h|}\ll X^{\frac{2627349}{2666036}+\varepsilon},
\end{align*}
which completes the proof of Lemma \ref{type-2-three-variables-integral-5}.     $\hfill$
\end{proof}

\begin{lemma}\label{yuqujian}
 For $\alpha\in(\tau,1-\tau)$, there holds
\begin{equation*}
  S(\alpha)\ll P^{\frac{2627349}{2666036}+\varepsilon}.
\end{equation*}
\end{lemma}
\begin{proof}
  First, we have
\begin{equation*}
  S(\alpha)=\mathcal{U}(\alpha)+O(P^{1/2}),
\end{equation*}
where
\begin{equation*}
  \mathcal{U}(\alpha)=\sum_{n\leqslant P}\Lambda(n)e([n^c]\alpha).
\end{equation*}
By a splitting argument, it is sufficient to prove that, for $P^{\frac{9449}{10000}}\ll X\ll P$ and $\alpha\in(\tau,1-\tau)$, there holds
\begin{equation*}
  \mathcal{U}^*(\alpha):=\sum_{X<n\leqslant 2X}\Lambda(n)e([n^c]\alpha)\ll X^{\frac{2627349}{2666036}+\varepsilon}.
\end{equation*}
By Lemma \ref{Buriev-exp-zhankai} with $H=X^{\frac{38687}{2666036}}$, we have
\begin{align}\label{U^*-fenjie}
   & \mathcal{U}^*(\alpha)= \,\, \sum_{X<n\leqslant2X}\Lambda(n)e\big(n^c\alpha-\{n^c\}\alpha\big)=\sum_{X<n\leqslant2X}\Lambda(n)e\big(n^c\alpha\big)e\big(-\{n^c\}\alpha\big)
                         \nonumber \\
   = & \,\,\sum_{X<n\leqslant2X}\Lambda(n)e\big(n^c\alpha\big)\Bigg(\sum_{|h|\leqslant H}c_h(\alpha)e(hn^c)
             +O\bigg(\min\bigg(1,\frac{1}{H\|n^c\|}\bigg)\bigg)\Bigg)
                           \nonumber \\
   = & \,\, \sum_{|h|\leqslant H}c_h(\alpha)\sum_{X<n\leqslant2X}\Lambda(n)e\big((h+\alpha)n^c\big)
             +O\Bigg(\log X\cdot\sum_{X<n\leqslant2X}\min\bigg(1,\frac{1}{H\|n^c\|}\bigg)\Bigg).
\end{align}
 By Lemma \ref{Heath-Brown-extention} and Lemma \ref{yijie-exp-pair} with the exponential pair $(\kappa,\lambda)=AB(0,1)=(\frac{1}{6},\frac{2}{3})$,
 we derive that
\begin{align}\label{U^*-fenjie-1}
    &\,\,   \sum_{X<n\leqslant2X}\min\bigg(1,\frac{1}{H\|n^c\|}\bigg)
                        \nonumber \\
   =& \,\, \sum_{X<n\leqslant2X}\sum_{\ell=-\infty}^{+\infty}a_\ell e(\ell n^c)
       \ll \sum_{\ell=-\infty}^{+\infty}\big|a_\ell\big|\Bigg|\sum_{X<n\leqslant2X}e(\ell n^c)\Bigg|
                        \nonumber \\
   \ll & \,\, \frac{X\log2H}{H}+\sum_{1\leqslant\ell\leqslant H}\frac{1}{\ell}\Bigg|\sum_{X<n\leqslant2X}e(\ell n^c)\Bigg|
              +\sum_{\ell>H}\frac{H}{\ell^2}\Bigg|\sum_{X<n\leqslant2X}e(\ell n^c)\Bigg|
                        \nonumber \\
   \ll & \,\, \frac{X\log2H}{H}+\sum_{1\leqslant\ell\leqslant H}\frac{1}{\ell}\bigg(\big(X^{c-1}\ell\big)^{\frac{1}{6}}X^{\frac{2}{3}}
                                  +\frac{1}{\ell X^{c-1}}\bigg)
               +\sum_{\ell>H}\frac{H}{\ell^2}\bigg(\big(X^{c-1}\ell\big)^{\frac{1}{6}}X^{\frac{2}{3}}+\frac{1}{\ell X^{c-1}}\bigg)
                         \nonumber \\
   \ll & \,\, X^{\frac{2627349}{2666036}}\log X+H^{\frac{1}{6}}X^{\frac{c}{6}+\frac{1}{2}}+X^{1-c}\ll X^{\frac{2627349}{2666036}}\log X.
\end{align}
Taking $U=X^{\frac{38687}{1333018}},V=X^{\frac{11958325}{23994324}}$, and $Z=[X^{\frac{1294331}{2666036}}]+\frac{1}{2}$ in
Lemma \ref{Heath-Brown-exponent-sum-fenjie}, it is
easy to see that the sum
\begin{equation*}
   \sum_{|h|\leqslant H}c_h(\alpha)\sum_{X<n\leqslant2X}\Lambda(n)e\big((h+\alpha)n^c\big)
\end{equation*}
can be represented as $O(\log^{10}X)$ sums, each of which either of Type I
\begin{equation*}
  \mathcal{S}_I(\alpha)= \sum_{|h|\leqslant H}c_h(\alpha)\sum_{M<m\leqslant2M}a(m)\sum_{K<k\leqslant2K}e\big((h+\alpha)(mk)^c\big)
\end{equation*}
with $K\gg Z,a(m)\ll m^\varepsilon, MK\asymp X$, or of Type II
\begin{equation*}
 \mathcal{S}_{II}(\alpha)= \sum_{|h|\leqslant H}c_h(\alpha)\sum_{M<m\leqslant2M}a(m)\sum_{K<k\leqslant2K}b(k)e\big((h+\alpha)(mk)^c\big)
\end{equation*}
with $U\ll M\ll V,a(m)\ll m^\varepsilon, b(k)\ll k^\varepsilon, MK\asymp X$. For the Type I sums, by noting the fact that $K\gg Z$ and $MK\asymp X$, we
deduce that $M\ll X^{\frac{1371705}{2666036}}$. From Lemma \ref{type-1-three-variables-integral-5}, we
have $\mathcal{S}_I(\alpha)\ll X^{\frac{2627349}{2666036}+\varepsilon}$. For the Type II sums, by Lemma \ref{type-2-three-variables-integral-5}, we have
$\mathcal{S}_{II}(\alpha)\ll X^{\frac{2627349}{2666036}+\varepsilon}$. Therefore, we conclude that
\begin{equation}\label{U^*-fenjie-2}
\sum_{|h|\leqslant H}c_h(\alpha)\sum_{X<n\leqslant2X}\Lambda(n)e\big((h+\alpha)n^c\big)\ll X^{\frac{2627349}{2666036}+\varepsilon}.
\end{equation}
From (\ref{U^*-fenjie})--(\ref{U^*-fenjie-2}), we complete the proof of Lemma \ref{yuqujian}. $\hfill$
\end{proof}

\begin{lemma}\label{T(alpha,X)-estimate-5}
For $\alpha\in(0,1),c\not\in\mathbb{Z}$, we have
\begin{equation*}
  \mathcal{T}(\alpha,X)\ll X^{\frac{c+4}{7}}\log X+\frac{1}{\alpha X^{c-1}}.
\end{equation*}
\end{lemma}
\begin{proof}
  Taking $H_1=X^{\frac{3-c}{7}}$, and by Lemma \ref{Buriev-exp-zhankai}, we deduce that
\begin{align}\label{T(alpha,X)-asymp}
        &\,\, \mathcal{T}(\alpha,X)= \sum_{X<n\leqslant2X}e\big((n^c-\{n^c\})\alpha\big)
                     \nonumber \\
    =  & \,\, \sum_{X<n\leqslant2X}e(n^c\alpha)\Bigg(\sum_{|h|\leqslant H_1}c_h(\alpha)e(hn^c) +O\bigg(\min\bigg(1,\frac{1}{H_1\|n^c\|}\bigg)\bigg)\Bigg)
                     \nonumber \\
    =  & \,\, \sum_{|h|\leqslant H_1}c_h(\alpha)\sum_{X<n\leqslant2X}e((h+\alpha)n^c)+O\Bigg(\sum_{X<n\leqslant2X}\min\bigg(1,\frac{1}{H_1\|n^c\|}\bigg)\Bigg).
\end{align}
From Lemma \ref{Heath-Brown-extention}, we get
\begin{equation}\label{T(alpha,X)-yuxiang}
     \sum_{X<n\leqslant2X}\min\bigg(1,\frac{1}{H_1\|n^c\|}\bigg)= \sum_{X<n\leqslant2X}\sum_{k=-\infty}^{+\infty}a_k e(kn^c)
 \ll \sum_{k=-\infty}^{+\infty}\big|a_k\big|\Bigg|\sum_{X<n\leqslant2X}e(k n^c)\Bigg|.
\end{equation}
 Then we shall use Lemma \ref{yijie-exp-pair} with the exponential pair $(\kappa,\lambda)=AB(0,1)=(\frac{1}{6},\frac{2}{3})$ to
estimate the sum over $n$ on the
right--hand side in (\ref{T(alpha,X)-yuxiang}), and derive that
\begin{align}\label{T(alpha,X)-asymp-1}
    &\,\,   \sum_{X<n\leqslant2X}\min\bigg(1,\frac{1}{H_1\|n^c\|}\bigg)
                        \nonumber \\
   \ll & \,\, \frac{X\log2H_1}{H_1}+\sum_{1\leqslant k \leqslant H_1}\frac{1}{k}\Bigg|\sum_{X<n\leqslant2X}e(k n^c)\Bigg|
              +\sum_{k>H_1}\frac{H_1}{k^2}\Bigg|\sum_{X<n\leqslant2X}e(kn^c)\Bigg|
                        \nonumber \\
   \ll & \,\, \frac{X\log2H_1}{H_1}+\sum_{1\leqslant k\leqslant H_1}
              \frac{1}{k}\bigg(\big(X^{c-1}k\big)^{\frac{1}{6}}X^{\frac{2}{3}}+\frac{1}{k X^{c-1}}\bigg)
                         \nonumber \\
    & \,\,  \qquad \qquad\quad +\sum_{k>H_1}\frac{H_1}{k^2}\bigg(\big(X^{c-1}k\big)^{\frac{1}{6}}X^{\frac{2}{3}}+\frac{1}{k X^{c-1}}\bigg)
                         \nonumber \\
   \ll & \,\, X^{\frac{c+4}{7}}\log X+H_1^{\frac{1}{6}}X^{\frac{c}{6}+\frac{1}{2}}+X^{1-c}
              \ll X^{\frac{c+4}{7}}\log X.
\end{align}
Similarly, for the first term in (\ref{T(alpha,X)-asymp}), we have
\begin{align}\label{T(alpha,X)-asymp-2}
       &\,\,   \sum_{|h|\leqslant H_1}c_h(\alpha)\sum_{X<n\leqslant2X}e\big((h+\alpha)n^c\big)
                     \nonumber \\
    =  & \,\, c_0(\alpha)\sum_{X<n\leqslant2X}e(\alpha n^c)+\sum_{1\leqslant|h|\leqslant H_1}c_h(\alpha)\sum_{X<n\leqslant2X}e\big((h+\alpha)n^c\big)
                     \nonumber \\
   \ll & \,\, \frac{1}{\alpha X^{c-1}}+\sum_{1\leqslant|h|\leqslant H_1}\frac{1}{h}\bigg(\big((h+\alpha)X^{c-1}\big)^{\frac{1}{6}}
               X^{\frac{2}{3}}+\frac{1}{(h+\alpha)X^{c-1}}\bigg)
                      \nonumber \\
   \ll & \,\, \frac{1}{\alpha X^{c-1}}+H_1^{\frac{1}{6}}X^{\frac{c}{6}+\frac{1}{2}}+X^{1-c}
                      \nonumber \\
   \ll & \,\, \frac{1}{\alpha X^{c-1}}+X^{\frac{c+4}{7}}\log X.
\end{align}
Combining (\ref{T(alpha,X)-asymp})--(\ref{T(alpha,X)-asymp-2}), we complete the
proof of Lemma \ref{T(alpha,X)-estimate-5}.   $\hfill$
\end{proof}

\section{Proof of Theorem \ref{Theorem-five-variables-integral} }
By the definition of $\mathscr{R}_5(N)$, it is easy to see that
\begin{align}\label{R_5(N)=R_5(1)(N)+R_5(2)(N)}
  \mathscr{R}_5(N)=& \,\, \int_0^1S^5(\alpha)e(-N\alpha)\mathrm{d}\alpha=\int_{-\tau}^{1-\tau}S^5(\alpha)e(-N\alpha)\mathrm{d}\alpha
                                      \nonumber \\
                =& \,\, \int_{-\tau}^{\tau}S^5(\alpha)e(-N\alpha)\mathrm{d}\alpha +\int_{\tau}^{1-\tau}S^5(\alpha)e(-N\alpha)\mathrm{d}\alpha
                                      \nonumber \\
                =& \,\, \mathscr{R}_5^{(1)}(N)+\mathscr{R}_5^{(2)}(N),
\end{align}
say. In order to prove Theorem \ref{Theorem-five-variables-integral}, we need the two following propositions, whose proofs will be given in the following
two subsections.
\begin{proposition}\label{Proposition-R_5(1)}
  For $1<c<\frac{4109054}{1999527},c\not=2$, there holds
\begin{equation*}
   \mathscr{R}_5^{(1)}(N)=\frac{\Gamma^5(1+1/c)}{\Gamma(5/c)}N^{5/c-1}+O\Big(N^{5/c-1}\exp\big(-(\log N)^{1/4}\big)\Big).
\end{equation*}
\end{proposition}

\begin{proposition}\label{Proposition-R_5(2)}
  For $1<c<\frac{4109054}{1999527},c\not=2$, there holds
\begin{equation*}
   \mathscr{R}_5^{(2)}(N)\ll N^{5/c-1-\varepsilon}.
\end{equation*}
\end{proposition}

From Proposition  \ref{Proposition-R_5(1)} and Proposition  \ref{Proposition-R_5(2)}, we obtain the result of Theorem \ref{Theorem-five-variables-integral}.

\subsection{Proof of Proposition \ref{Proposition-R_5(1)}}
In this subsection, we shall concentrate on establishing Proposition  \ref{Proposition-R_5(1)}. Define
\begin{align*}
   & G(\alpha)=\sum_{m\leqslant N}\frac{1}{c}m^{\frac{1}{c}-1}e(m\alpha),\\
   & \mathscr{H}_1(N)=\int_{-\tau}^\tau G^5(\alpha)e(-N\alpha)\mathrm{d}\alpha, \\
   & \mathscr{H}(N)=\int_{-\frac{1}{2}}^{\frac{1}{2}} G^5(\alpha)e(-N\alpha)\mathrm{d}\alpha.
\end{align*}
Then we can write
\begin{equation}\label{R_5^{(1)}(N)-fenjie}
   \mathscr{R}_5^{(1)}(N)=\big(\mathscr{R}_5^{(1)}(N)-\mathscr{H}_1(N)\big)+\big(\mathscr{H}_1(N)-\mathscr{H}(N)\big)+\mathscr{H}(N).
\end{equation}
As is shown in Theorem 2.3 of Vaughan \cite{Vaughan-book}, we derive that
\begin{equation}\label{H(N)-asymp}
   \mathscr{H}(N)=\frac{\Gamma^5(1+1/c)}{\Gamma(5/c)}P^{5-c}+O(P^{4-c}).
\end{equation}
By Lemma 2.8 of  Vaughan \cite{Vaughan-book}, we know that
\begin{equation}\label{H_1-H}
   \mathscr{H}_1(N)-\mathscr{H}(N)\ll \int_\tau^{\frac{1}{2}}\big|G(\alpha)\big|^5\mathrm{d}\alpha\ll\int_\tau^{\frac{1}{2}}\alpha^{-\frac{5}{c}}\mathrm{d}\alpha
   \ll\tau^{1-\frac{5}{c}}\ll P^{5-c-\nu}
\end{equation}
for some $\nu>0$. Next, we consider the estimate of $|\mathscr{R}_5^{(1)}(N)-\mathscr{H}_1(N)|$. We have
\begin{align}\label{R_5(1)(N)-H_1(N)}
       & \,\,  \mathscr{R}_5^{(1)}(N)-\mathscr{H}_1(N)
                \ll \int_{-\tau}^\tau\big|S^5(\alpha)-G^5(\alpha)\big|\mathrm{d}\alpha
                \nonumber \\
  \ll & \,\,\int_{-\tau}^\tau\big|S(\alpha)-G(\alpha)\big|\big(|S(\alpha)|^4+|G(\alpha)|^4\big)\mathrm{d}\alpha
                \nonumber \\
  \ll & \,\, \sup_{|\alpha|\leqslant\tau}\big|S(\alpha)-G(\alpha)\big|\times\bigg(\int_{-\tau}^\tau|S(\alpha)|^4\mathrm{d}\alpha+
              \int_{-\frac{1}{2}}^{\frac{1}{2}}|G(\alpha)|^4\mathrm{d}\alpha\bigg).
\end{align}
From Lemma 2.8 of Vaughan \cite{Vaughan-book}, we know that
\begin{equation*}
   G(\alpha)\ll \min\big(N^{\frac{1}{c}},|\alpha|^{-\frac{1}{c}} \big).
\end{equation*}
Therefore, there holds
\begin{align}\label{G(alpha)-fourth-0-1}
   & \int_{-\frac{1}{2}}^{\frac{1}{2}}\big|G(\alpha)\big|^4\mathrm{d}\alpha\ll \int_{0}^{\frac{1}{2}}
             \min\big(N^{\frac{1}{c}},|\alpha|^{-\frac{1}{c}} \big)^4\mathrm{d}\alpha
                   \nonumber \\
   \ll & \,\, \int_{0}^{\frac{1}{N}}N^{\frac{4}{c}}\mathrm{d}\alpha+\int_{\frac{1}{N}}^{\frac{1}{2}}\alpha^{-\frac{4}{c}}\mathrm{d}\alpha
          \ll N^{\frac{4}{c}-1}\ll P^{4-c}.
\end{align}
For $|\alpha|\leqslant\tau$,
 from Lemma \ref{S^*(alpha)-fourth-power}, we obtain
\begin{equation}\label{S(alpha)-fourth-major}
  \int_{-\tau}^\tau\big|S(\alpha)\big|^4\mathrm{d}\alpha\ll P^{4-c}\log^6P.
\end{equation}
Finally, we consider the upper bound of $\big|S(\alpha)-G(\alpha)\big|$ under the condition $|\alpha|\leqslant\tau$.
Trivially, we have
\begin{align}\label{S(alpha)-tihuan}
   S(\alpha) = & \,\, \sum_{p\leqslant P}(\log p)e\big(p^{c}\alpha\big)+O(\tau P)=\sum_{n\leqslant P}\Lambda(n)e(n^c\alpha)+O(P^{1/2})+O(\tau P)
                          \nonumber \\
   = &\,\, \sum_{n\leqslant P}\Lambda(n)e(n^c\alpha)+O(P^{1-\varepsilon}).
\end{align}
From Lemma \ref{Titchmarsh-lemma4.8}, we know that, for $|\alpha|\leqslant\tau$ and $u\geqslant2$, there holds
\begin{equation*}
  \sum_{1<m\leqslant u}e(m\alpha)=\int_1^u e(t\alpha)\mathrm{d}t+O(1).
\end{equation*}
By partial summation and the above identity, we deduce that
\begin{align}\label{yihuan-G(alpha)}
        \sum_{n\leqslant P}\Lambda(n)e(n^c\alpha)
   = & \,\, \int_1^Pe(t^c\alpha)\mathrm{d}\bigg(\sum_{n\leqslant t}\Lambda(n)\bigg)=\int_1^Pe(t^c\alpha)\mathrm{d}t+O\big(P\exp\big(-(\log P)^{1/3}\big)\big)
                     \nonumber \\
   = & \,\, \int_1^N\frac{1}{c}u^{\frac{1}{c}-1}e(u\alpha)\mathrm{d}u+O\big(P\exp\big(-(\log P)^{1/3}\big)\big)
                    \nonumber \\
   = & \,\, \int_1^N\frac{1}{c}u^{\frac{1}{c}-1}\mathrm{d}\bigg(\int_1^u e(t\alpha)\mathrm{d}t\bigg)+O\big(P\exp\big(-(\log P)^{1/3}\big)\big)
                     \nonumber \\
   = & \,\, \int_1^N\frac{1}{c}u^{\frac{1}{c}-1}\mathrm{d}\bigg( \sum_{1<m\leqslant u}e(m\alpha)+O(1) \bigg)+O\big(P\exp\big(-(\log P)^{1/3}\big)\big)
                    \nonumber \\
   = & \,\, \sum_{m\leqslant N}\frac{1}{c}m^{\frac{1}{c}-1}e(m\alpha)+O\big(P\exp\big(-(\log P)^{1/3}\big)\big)
                     \nonumber \\
   = & \,\, G(\alpha)+ O\big(P\exp\big(-(\log P)^{1/3}\big)\big).
\end{align}
From (\ref{S(alpha)-tihuan}) and (\ref{yihuan-G(alpha)}), we deduce that
\begin{equation}\label{S(alpha)-G(alpha)-upper}
  \sup_{|\alpha|\leqslant\tau}\big|S(\alpha)-G(\alpha)\big|\ll P\exp\big(-(\log P)^{1/3}\big).
\end{equation}
inserting (\ref{G(alpha)-fourth-0-1}), (\ref{S(alpha)-fourth-major}) and (\ref{S(alpha)-G(alpha)-upper}) into (\ref{R_5(1)(N)-H_1(N)}), we get
\begin{equation}\label{R_5(1)(N)-H_1(N)-final-upper}
  \mathscr{R}_5^{(1)}(N)-\mathscr{H}_1(N)\ll P^{5-c}\exp\big(-(\log P)^{1/4}\big).
\end{equation}
By (\ref{R_5^{(1)}(N)-fenjie})--(\ref{H_1-H}) and (\ref{R_5(1)(N)-H_1(N)-final-upper}), we obtain the desired result of Proposition \ref{Proposition-R_5(1)}.

\subsection{Proof of Proposition \ref{Proposition-R_5(2)}}
In this subsection, we devote to prove Proposition  \ref{Proposition-R_5(2)}. First, we have
\begin{equation}\label{S(alpha)-fenjie}
  S(\alpha)=\sum_{p\leqslant P^{\frac{9449}{10000}}}(\log p)e\big([p^c]\alpha\big)+\sum_{P^{\frac{9449}{10000}}<p\leqslant P}(\log p)e\big([p^c]\alpha\big).
\end{equation}
By a splitting argument, (\ref{S(alpha)-fenjie}) and Lemma \ref{S(alpha)-fourth-power}, we deduce that
\begin{align}\label{R_5^(2)(N)-erfenfenjie}
  \mathscr{R}_5^{(2)}(N)\ll & \,\, (\log P)\max_{P^{\frac{9449}{10000}}\ll X\ll P}\bigg|\int_\tau^{1-\tau}S^4(\alpha)S(\alpha,X)e(-N\alpha)\mathrm{d}\alpha\bigg|
                              +P^{\frac{9449}{10000}}\int_0^1\big|S(\alpha)\big|^4\mathrm{d}\alpha
                               \nonumber \\
   \ll & \,\, (\log P)\max_{P^{\frac{9449}{10000}}\ll X\ll P}\bigg|\int_\tau^{1-\tau}S^4(\alpha)S(\alpha,X)e(-N\alpha)\mathrm{d}\alpha\bigg|+P^{\frac{9449}{10000}+\varepsilon}(P^{4-c}+P^2)
                               \nonumber \\
   \ll & \,\, (\log P)\max_{P^{\frac{9449}{10000}}\ll X\ll P}\bigg|\int_\tau^{1-\tau}S^4(\alpha)S(\alpha,X)e(-N\alpha)\mathrm{d}\alpha\bigg|+P^{5-c-\varepsilon}.
\end{align}
For $P^{\frac{9449}{10000}}\ll X\ll P$, we have
\begin{align*}
  & \,\, \bigg|\int_\tau^{1-\tau}S^4(\alpha)S(\alpha,X)e(-N\alpha)\mathrm{d}\alpha\bigg|
                      \nonumber \\
  = & \,\, \Bigg|\sum_{X<p\leqslant2X}(\log p)\int_\tau^{1-\tau}S^4(\alpha)e\big(([p^c]-N)\alpha\big)\mathrm{d}\alpha\Bigg|
                       \nonumber \\
 \leqslant & \,\, \sum_{X<p\leqslant2X}(\log p)\Bigg|\int_\tau^{1-\tau}S^4(\alpha)e\big(([p^c]-N)\alpha\big)\mathrm{d}\alpha \Bigg|
                       \nonumber \\
 \ll & \,\, (\log X)\sum_{X<n\leqslant2X}\Bigg|\int_\tau^{1-\tau}S^4(\alpha)e\big(([n^c]-N)\alpha\big)\mathrm{d}\alpha \Bigg|.
\end{align*}
By Cauchy's inequality, we deduce that
\begin{align}\label{yuqujian-cauchy-upper-5}
  & \,\, \bigg|\int_\tau^{1-\tau}S^4(\alpha)S(\alpha,X)e(-N\alpha)\mathrm{d}\alpha\bigg|
                      \nonumber \\
  \ll & \,\, X^{\frac{1}{2}+\varepsilon}\Bigg(\sum_{X<n\leqslant2X}\bigg|\int_\tau^{1-\tau}S^4(\alpha)
              e\big(([n^c]-N)\alpha\big)\mathrm{d}\alpha\bigg|^2\Bigg)^{\frac{1}{2}}
                       \nonumber \\
  = & \,\, X^{\frac{1}{2}+\varepsilon}\Bigg(\sum_{X<n\leqslant2X}\int_\tau^{1-\tau}S^4(\alpha)e\big(([n^c]-N)\alpha\big)\mathrm{d}\alpha\cdot
            \int_\tau^{1-\tau}\overline{S^4(\beta)e\big(([n^c]-N)\beta\big)}\mathrm{d}\beta\Bigg)^{\frac{1}{2}}
                        \nonumber \\
  = & \,\, X^{\frac{1}{2}+\varepsilon} \bigg(\int_\tau^{1-\tau}\overline{S^4(\beta)e(-N\beta)}\mathrm{d}\beta
           \int_{\tau}^{1-\tau}S^4(\alpha)\mathcal{T}(\alpha-\beta,X)e(-N\alpha)\mathrm{d}\alpha\bigg)^{\frac{1}{2}}
                        \nonumber \\
  \ll & \,\, X^{\frac{1}{2}+\varepsilon}\bigg(\int_\tau^{1-\tau}\big|S(\beta)\big|^4\mathrm{d}\beta
               \int_\tau^{1-\tau}\big|S(\alpha)\big|^4 \big|\mathcal{T}(\alpha-\beta,X)\big|\mathrm{d}\alpha\bigg)^{\frac{1}{2}}.
\end{align}
For the inner integral in (\ref{yuqujian-cauchy-upper-5}), we have
\begin{align}\label{inner-jifen-fenjie-5}
  & \,\, \int_\tau^{1-\tau}\big|S(\alpha)\big|^4\big|\mathcal{T}(\alpha-\beta,X)\big|\mathrm{d}\alpha
                       \nonumber \\
  \ll & \,\, \Bigg(\int_{(\tau,1-\tau)\cap\{\alpha:\,|\alpha-\beta|\leqslant X^{-c}\}}
             +\int_{(\tau,1-\tau)\cap\{\alpha:\,|\alpha-\beta|> X^{-c}\}} \Bigg)
             \big|S^4(\alpha)\mathcal{T}(\alpha-\beta,X)\big|\mathrm{d}\alpha.
\end{align}
For the first term on the right--hand side of (\ref{inner-jifen-fenjie-5}), we use Lemma \ref{yuqujian} and the trivial
estimate $\mathcal{T}(\alpha-\beta,X)\ll X$ to deduce that
\begin{align}\label{inner-jifen-fenjie-1-5}
  & \,\, \int_{(\tau,1-\tau)\cap\{\alpha:\,|\alpha-\beta|\leqslant X^{-c}\}}\big|S^4(\alpha)\mathcal{T}(\alpha-\beta,X)\big|\mathrm{d}\alpha
                       \nonumber \\
  \ll & \,\, X\cdot\sup_{\alpha\in(\tau,1-\tau)}\big|S(\alpha)\big|^4\times\int_{|\alpha-\beta|\leqslant X^{-c}}\mathrm{d}\alpha
               \ll P^{\frac{2627349}{666509}+\varepsilon}X^{1-c}.
\end{align}
For the second term on the right--hand side of (\ref{inner-jifen-fenjie-5}), by Lemma \ref{yuqujian} and Lemma \ref{T(alpha,X)-estimate-5}, we obtain
\begin{align}\label{inner-jifen-fenjie-2-5}
  & \,\, \int_{(\tau,1-\tau)\cap\{\alpha:\,|\alpha-\beta|> X^{-c}\}}\big|S^4(\alpha)\mathcal{T}(\alpha-\beta,X)\big|\mathrm{d}\alpha
                       \nonumber \\
  \ll & \,\, \int_{(\tau,1-\tau)\cap\{\alpha:\,|\alpha-\beta|> X^{-c}\}}
             \big|S(\alpha)\big|^4\bigg(X^{\frac{c+4}{7}}\log X+\frac{1}{|\alpha-\beta|X^{c-1}}\bigg)\mathrm{d}\alpha
                       \nonumber \\
  \ll & \,\, X^{\frac{c+4}{7}+\varepsilon}\times\int_0^1|S(\alpha)|^4\mathrm{d}\alpha
             +\sup_{\alpha\in(\tau,1-\tau)}\big|S(\alpha)\big|^4\times\int_{|\alpha-\beta|>X^{-c}}\frac{\mathrm{d}\alpha}{|\alpha-\beta|X^{c-1}}
                       \nonumber \\
  \ll & \,\, X^{\frac{c+4}{7}+\varepsilon}\big(P^{4-c}+P^2\big)P^\varepsilon+P^{\frac{2627349}{666509}+\varepsilon}X^{1-c}.
\end{align}
Combining (\ref{inner-jifen-fenjie-5})--(\ref{inner-jifen-fenjie-2-5}), we conclude that
\begin{equation}\label{inner-upper-bound-5}
  \int_\tau^{1-\tau}\big|S^4(\alpha)\mathcal{T}(\alpha-\beta,X)\big|\mathrm{d}\alpha\ll X^{\frac{c+4}{7}+\varepsilon}
                      \big(P^{4-c}+P^2\big) P^\varepsilon+P^{\frac{2627349}{666509}+\varepsilon}X^{1-c}.
\end{equation}
Inserting (\ref{inner-upper-bound-5}) into (\ref{yuqujian-cauchy-upper-5}), we obtain
\begin{align*}
    & \,\, \bigg|\int_\tau^{1-\tau}S^4(\alpha)S(\alpha,X)e(-N\alpha)\mathrm{d}\alpha\bigg|
                       \nonumber \\
 \ll & \,\,  X^{\frac{1}{2}+\varepsilon}\bigg(\Big(X^{\frac{c+4}{7}+\varepsilon}\big(P^{4-c}+P^2\big) P^\varepsilon
             +P^{\frac{2627349}{666509}+\varepsilon}X^{1-c}\Big)\big(P^{4-c}+P^2\big)P^\varepsilon\bigg)^{\frac{1}{2}}
                       \nonumber \\
 \ll & \,\,  X^{\frac{c+11}{14}}\big(P^{4-c}+P^2\big)P^\varepsilon+P^{\frac{2627349}{1333018}+\varepsilon}\big(P^{\frac{4-c}{2}}+P\big)X^{1-\frac{c}{2}}.
\end{align*}
For $1<c<2$, we have
\begin{align}\label{last-c-1-2-upper}
    & \,\, \bigg|\int_\tau^{1-\tau}S^4(\alpha)S(\alpha,X)e(-N\alpha)\mathrm{d}\alpha\bigg|
                       \nonumber \\
 \ll & \,\,  P^{\frac{c+11}{14}}\cdot P^{4-c+\varepsilon}+P^{\frac{2627349}{1333018}+\varepsilon}\cdot P^{\frac{4-c}{2}}\cdot P^{1-\frac{c}{2}}
 \ll  P^{5-c-\varepsilon}.
\end{align}
For $2<c<\frac{4109054}{1999527}$, we have
\begin{align}\label{last-c->2-upper}
    & \,\, \bigg|\int_\tau^{1-\tau}S^4(\alpha)S(\alpha,X)e(-N\alpha)\mathrm{d}\alpha\bigg|
                       \nonumber \\
 \ll & \,\,  P^{\frac{c+11}{14}}\cdot P^{2+\varepsilon}+P^{\frac{2627349}{1333018}+\varepsilon}\cdot P\cdot P^{\frac{9449}{10000}(1-\frac{c}{2})}
 \ll  P^{5-c-\varepsilon}.
\end{align}
From (\ref{R_5^(2)(N)-erfenfenjie}), (\ref{last-c-1-2-upper}) and (\ref{last-c->2-upper}), we deduce that
\begin{equation*}
  \mathscr{R}_5^{(2)}(N)\ll P^{5-c-\varepsilon}
\end{equation*}
provided that $1<c<\frac{4109054}{1999527},c\not=2$, which completes the proof  of Proposition \ref{Proposition-R_5(2)}.

\section*{Acknowledgement}

The authors would like to express the most sincere gratitude to the referee for his/her patience in refereeing this paper.


\begin{thebibliography}{99}

\bibitem{Arkhipov-Zhitkov-1984}G. I. Arkhipov, A. N. Zhitkov, \textit{On Waring¡¯s problem with non-integer degrees},
                                   Izv. Akad. Nauk SSSR, \textbf{48} (1984), 1138--1150.


\bibitem{Buriev-thesis-1989}K. Buriev, \textit{Additive problems with prime numbers}, Thesis, Moscow University, 1989.

\bibitem{Cai-2018}Y. C. Cai, \textit{On a Diophantine equation involving primes}, Ramamujan J., DOI: 10.1007/s11139-018-0027-6.

\bibitem{Deshouillers-1973}J. M. Deshouillers, \textit{Probl\`{e}me de Waring avec exposants non entiers}, Bull. Soc. Math. France,
                                        \textbf{101} (1973), 285--295.

\bibitem{Deshouillers-1974}J. M. Deshouillers, \textit{Un probl\`{e}me binaire en th\'{e}orie additive}, Acta Arith., \textbf{25} (1974), no. 4, 393--403.


\bibitem{Fouvry-Iwaniec-1989}E. Fouvry, H. Iwaniec, \textit{Exponential sums with monomials}, J. Number Theory, \textbf{33} (1989), no. 3, 311--333.

\bibitem{Graham-Kolesnik-book}S. W. Graham, G. Kolesnik, \textit{Van der Corput's Method of Exponential Sums}, Cambridge University Press, New York, 1991.

\bibitem{Gritsenko-1993}S. A. Gritsenko, \textit{Three additive problems}, Russian Acad. Sci. Izv. Math., \textbf{41} (1993), no. 3, 447--464.

\bibitem{Heath-Brown-1983}D. R. Heath--Brown, \textit{The Pjatecki\u{\i}-\u{S}apiro prime number theorem},
                                                J. Number Theory, \textbf{16} (1983), no. 2, 242--266.

\bibitem{Hua-1938}L. K. Hua, \textit{Some results in the additive prime number theory}, Quart. J. Math. Oxford Ser. (2), \textbf{9} (1938),  no. 1, 68--80.


\bibitem{Kumchev-Nedeva-1998}A. Kumchev, T. Nedeva, \textit{On an equation with prime numbers}, Acta Arith., \textbf{83} (1998), no. 2, 117--126.

\bibitem{Laporta-Tolev-1995}M. B. S. Laporta, D. I. Tolev, \textit{On an equation with prime numbers}, Math. Notes, \textbf{57} (1995), no. 5--6, 654--657.

\bibitem{Robert-Sargos-2006} O. Robert, P. Sargos, \textit{Three--dimensional exponential sums with monomials},
                                       J. Reine Angew. Math., \textbf{591} (2006), 1--20.

\bibitem{Segal-1933-1}B. I. Segal,  \textit{On a theorem similar to the Waring theorem}, Dokl. Akad. Nauk. SSSR, \textbf{1} (1933), 47--49.

\bibitem{Segal-1933-2}B. I. Segal,  \textit{The Waring theorem with fractional and irrational degrees}, Trudy Mat. Inst. Steklov.,  \textbf{5} (1934), 73--86.

\bibitem{Titchmarsh-book}E. C. Titchmarsh, \textit{The Theory of the Riemann Zeta--Function, 2nd edn.}, (Revised by D. R. Heath--Brown),
                                            Oxford University Press, Oxford, 1986.

\bibitem{Vaughan-book}R. C. Vaughan, \textit{The Hardy--Littlewood Method, 2nd edn.}, Cambridge University Press, Cambridge, 1997.

\bibitem{Vinogradov-1937}I. M. Vinogradov, \textit{Representation of an odd number as the sum of three primes},
                                                Dokl. Akad. Nauk. SSSR, \textbf{15} (1937), 291--294.

\bibitem{Zhai-Cao-2002}W. G. Zhai, X. D. Cao, \textit{A Diophantine equation with prime numbers}, Acta Math. Sinica (Chin. Ser.),
                                            \textbf{45} (2002), no. 3, 443--454.











































\end{thebibliography}
\end{document}